\documentclass[12pt,a4paper]{article}

\usepackage{pdflscape}
\usepackage{amsmath}
\usepackage{amssymb}
\usepackage{latexsym}
\usepackage{srcltx}
\usepackage{color}
\usepackage{epsfig}
\usepackage{color}

\usepackage{pstricks}
\usepackage{graphicx,psfrag}
\usepackage{amsthm}

\newcommand{\re}{{\mathbb R}}

\newcommand{\cU}{{\mathcal{U}}}

\newcommand{\bz}{{\boldsymbol{z}}}
\newcommand{\be}{{\boldsymbol{e}}}
\newcommand{\bp}{{\boldsymbol{p}}}
\newcommand{\bh}{{\boldsymbol{h}}}

\newcommand{\bc}{{\boldsymbol{c}}}
\newcommand{\bd}{{\boldsymbol{d}}}

\newcommand {\al}   {\alpha}          
                
      \newcommand {\del}  {\delta}          
      \newcommand {\ve}   {\varepsilon}        \newcommand {\vphi} {\varphi}
               
      \newcommand {\om}   {\omega}          \newcommand {\Om}  {\Omega}
      \newcommand {\pl}   {\partial}
     \newcommand {\beq}  {\begin{equation}}  \newcommand {\beqo}  {\begin{equation*}}
      \newcommand {\eeq}  {\end{equation}}   \newcommand {\eeqo}  {\end{equation*}}

\newtheorem{theorem}{Theorem}
\newtheorem{prop}{Proposition}
\newtheorem{lemma}{Lemma}
\newtheorem{cor}{Corollary}
\newtheorem{remark}{Remark}
\newtheorem{ex}{Example}

\newtheorem{prob}{Problem}

\date{}

\title{On inequalities between norms of partial derivatives\\
on convex domains}

\author{Alexander Plakhov\thanks{Center for R{\&}D in Mathematics and Applications, Department of Mathematics, University of Aveiro, Portugal and Institute for Information Transmission Problems, Moscow, Russia, plakhov@ua.pt} \and Vladimir Protasov\thanks{DISIM, University of L'Aquila, Italy}}

\begin{document}
\maketitle

\begin{abstract}

 We consider
 inequalities between $L_p$-norms of partial derivatives, $p\in [1,+\infty]$,  for bivariate concave
 functions on a convex domain that vanish on the boundary. Can the ratio between
  those norms be arbitrarily large? If not, what is the upper bound?
We show that for~$p=1$, the ratio is always bounded and find sharp
estimates, while for~$p> 1$,  the answer depends on the geometry of the domain.
\end{abstract}

\begin{quote}
\end{quote}

\begin{quote}
\end{quote}

\begin{center}
\large{\textbf{1. Introduction}}
\end{center}
\bigskip

We consider a convex compact domain~$\Omega \subset \re^2$ and the class~$\cU_{\, \Omega}$ of continuous concave functions $u(x,y)$
on~$\Omega$  such that $u\rfloor_{\, \partial \Omega} = 0$ and the
gradient~$\nabla u(x,y) \, = \, (u_x, u_y)$
is uniformly bounded over all points~$(x,y) \, \in \, int \, \Omega$ where it is defined. Since concave continuous functions are absolutely
continuous on~$\Omega$, the gradient is defined a.e. and belongs to~$L_p(\Omega)$
for all $p\in [1, +\infty]$. For arbitrary~$f\in L_p(\Omega)$, we use the standard notation
$$
\bigl\|f\bigr\|_p \ = \
\left(\int_{\Omega}|f(x,y)|^p\, dx \, dy\right)^{1/p},\ p\in [1, +\infty); \qquad
 \bigl\|f\bigr\|_{\infty} \ = \
{\rm ess \,sup}_{(x, y) \in \Omega}|f(x,y)|.
$$
 We are interested in the value
\begin{equation}\label{eq.10-0}
K_p \ = \ K_p(\Omega) \quad  = \quad  \sup_{ \stackrel{u\in\cU_\Om}{u\ne 0}} \ \frac{\|u_x\|_p}{\|u_y\|_p} \, .
\end{equation}
It will be shown that  in the case~$p=1$ the value $K_1(\Omega)$
is always finite and can be efficiently estimated.
For~$p>1$, the value~$K_p(\Omega)$  can be infinite or finite depending on the geometry of the domain~$\Omega$. We obtain upper bounds for~$K_p$ in terms of
geometric characteristics of~$\Omega$. For~$p=\infty$, this value is found explicitly
for every domain.  The problem of computing the optimal
domains that minimize~$K_p$ under some natural assumptions is formulated.
It is interesting that for all~$p>1$, neither the Euclidean ball nor any other
domain with a smooth boundary is optimal (Theorems~\ref{th.10-p} and \ref{th.p=infty}).
On the other hand, in~$L_1$, the Euclidean ball is optimal,
along with infinitely many other optimal domains.

For~$p=2$, the problem of estimating~$K_p$ is applied to the study of the
surfaces of smallest resistance in the Newton aerodynamical problem~\cite{PP25}.
 Estimating the value~$K_p$
can also be attributed to the so-called inequalities between derivatives,
which is an important branch of the approximation theory known from the classical
results of Bernstein, Kolmogorov, Landau, Nikolsky, Calderon, Zygmund, etc. Such inequalities involve
norms of derivatives of various orders in different functional spaces.
 See~\cite{D89, HR34, M91, TM85}  for an extensive bibliography.

\begin{remark}\label{r10}
{\em Without the concavity condition on the function~$u$,
the value~$K_p$ is infinite for every convex domain~$\Omega$.
To see this it suffices to
take an arbitrary nonzero~$C^1$ function~$u_0(x,y)$ that vanishes on~$\, \partial \Omega$
and multiply it by a highly oscillating function~$\varphi(x)$, for example,
by~$\varphi_N(x) = \cos \,Nx$. Then
the ratio of norms of partial derivatives for the
function~$u_N(x,y) \, = \, u_0(x,y)\varphi_N(x)$
tends to infinity as~$N\to \infty$. }
\end{remark}

\begin{remark}\label{r20}
{\em  The value $K_p$ admits an intuitive interpretation. Suppose we are interested in measuring the average slope of a hill. The hill has a convex shape and rests on the flat base $\Omega$, hence its upper surface is the graph of a function $u \in \mathcal{U}_\Om$. Of course, it is reasonable to take $\| u \|_p$ as the average measure. Suppose, however, that we are limited in our actions and can only measure slope from north to south or from east to west, thus obtaining $\| u_x \|_p$ or $\| u_y \|_p$. The question is, how large is the discrepancy between these quantities, expressed by their ratio $K_p$?}
\end{remark}

The ratio $\frac{\| u_x \|_p}{\| u_y \|_p}$ can be made arbitrarily large or arbitrarily small by replacing $u(x,y)$ with $u(kx,y)$, where the coefficient $k$ tends to infinity or to zero. Accordingly, the domain of $u$ is uniformly stretched $k$ times along the $x$-axis. This means that the quantity $K_p(\Omega)$ can be made arbitrarily large and the quantity
\beq\label{k small}
k_p(\Omega) \quad  = \quad  \inf_{\stackrel{u\in\cU_\Om}{u\ne 0}} \ \frac{\|u_x\|_p}{\|u_y\|_p}
\eeq
can be made arbitrarily small just by stretching and compressing $\Om.$ However, the ratio $\frac{K_p(\Omega)}{k_p(\Omega)}$ is of interest. We pose the following Minimax problem:

\begin{prob}\label{pr.minimax} Find
$$
 r_p \ =\ \inf_\Om \frac{K_p(\Omega)}{k_p(\Omega)}.
$$
\end{prob}

In this paper we find a solution to this problem for $p=\infty$: \ $r_\infty = 1$, \, and a partial solution for $p=1$: \ $r_1 \le 4$. \, The solution for $1<p<\infty$ is unknown.

\begin{center}
\large{\textbf{2. The case~$\bp = 1$}}
\end{center}
\medskip

 We are going to show that~$K_{1}(\Omega) \, \le \, 2\, \frac{w_x}{w_y}$,
 where~$w_x, w_y$ are, respectively, the horizontal and vertical sides of the circumscribed
 rectangle of~$\Omega$. Moreover, this estimate  in general cannot be improved, i.e.,
it is sharp for some domains~$\Omega$. Sometimes we consider an extension
of~$\cU_{\Omega}$ to the class of arbitrary
concave functions on~$\Omega$ vanishing on~$\partial \Omega$.
In this case $u$ may be discontinuous on the boundary and its partial derivatives~$u_x, u_y$
are understood in the sense of distributions.

\begin{theorem}\label{th.5-1}
For every~$u\in \cU_{\, \Omega}$, we have
\begin{equation}\label{eq.5-1}
K_1\ = \ \sup_{u\in \cU_{\, \Omega}}\ \frac{\|u_x \|_1}{\|u_y \|_1}\quad \le \quad 2\ \frac{w_x}{w_y}\ ,
\end{equation}
where~$w_{x}, w_y$ are  the sidelengths  of the rectangle circumscribed around~$\Omega$
with the sides parallel to the~$x$ and~$y$-axes respectively.
Inequality~(\ref{eq.5-1}) becomes an equality  if and only if~$\Omega$ contains
 a vertical segment with ends on the horizontal sides of the circumscribed rectangle.
\end{theorem}

This theorem immediately gives an above estimate for Minimax problem \ref{pr.minimax}. Indeed,
\beq\label{estimateM}
\frac{K_1(\Omega)}{k_1(\Omega)}\ = \ \frac{\sup_{u\in \cU_{\, \Omega}}\ \frac{\|u_x \|_1}{\|u_y \|_1}}{\inf_{u\in \cU_{\, \Omega}}\ \frac{\|u_x \|_1}{\|u_y \|_1}}  \
=\ \sup_{u\in \cU_{\, \Omega}}\ \frac{\|u_x \|_1}{\|u_y \|_1} \, \sup_{u\in \cU_{\, \Omega}}\ \frac{\|u_y \|_1}{\|u_x \|_1}\
\le \ 2\ \frac{w_x}{w_y}\ 2\ \frac{w_y}{w_x}\ =\ 4.
\eeq
It follows that $r_1 \le 4.$

This theorem will be obtained as  a corollary of Theorems~\ref{th.10-1}
and~Theorem~\ref{th.20-1} proved below. Then in Theorem~\ref{th.30-1}
we estimate the ratio between the maximal and minimal $L_1$-norms
of directional derivatives~$\frac{\|u_{\bh_1}\|_1}{\|u_{\bh_2}\|_1}$ over all possible pairs of
unit vectors~$\bh_1, \bh_2$.

We begin with  auxiliary  results.
For arbitrary $t\in \re$, we denote by $\bigl[a(t), b(t)\bigr]$
the segment of intersection of~$\Omega$ with the horizontal line  $y=t$ and
$m(t)\, = \, \max\limits_{x\in [a,b]}u(x,t)$ on this segment.
If this line does not intersect~$\Omega$, then the segment is empty and $m=0$.
\medskip

\begin{lemma}\label{l.10-1}
 For every~$y$, we have~$\int_{a}^{b}\,
\bigl| u_x(x,y)\bigr|\, dx\, = 2m$, where $a,b, m$ are  $a(y),b(y)$ and $m(y)$
respectively.
\end{lemma}

\noindent {\tt Proof}. Since for every~$y$, the univariate function
$u(\cdot \, , \, y)$   is concave,
its derivative is non-increasing. Hence,
$u_x(x,y) \ge 0$  for~$x<\tau$ and $u_x(x,y) \le 0$ for~$x>\tau$, where
$x=\tau$ is a point of maximum. Therefore,
$\int_{a}^{\tau}|u_x(x,y)|\, dx \, = \, \int_{a}^{\tau}u_x(x,y)\, dx\, = \,
u(\tau, y) - u(a, y)= m - 0 = m$.
Similarly, $\int_{\tau}^{b}|u_x(x,y)|\, dx \, = \, \int_{\tau}^{b}- u_x (x,y)\, dx\, = \,
u(\tau, y) - u(b, y)= m - 0 = m$. Hence, $\int_{a}^{b}\,
| u_x(x,y)|\, dx\, = \, m+m \, = 2m$.

{\hfill $\Box$}
\smallskip

Lemma~\ref{l.10-1} immediately implies
\smallskip

\begin{prop}\label{p.10-1}  Let a segment $[c,d]$ of the axis $OY$ be the projection of the set~$\Omega$. Then  $\|u_x\|_1 \, = \, \int_{c}^d \, 2m(y)\, dy$.
\end{prop}

\smallskip

 Proposition~\ref{p.10-1} guarantees, in particular, that the supremum and infimum  of the value
 $\|u_x\|_1$
 over a bounded set of functions~$u$ are always attained in distributions.
Let us denote by~$w = d-c$ the length of the projection of~$\Omega$
to $OY$. Every point~$\bc \in \Omega$ with the second coordinate equal to~$c$
is called~{\em lower}, a point~$\bd \in \Omega$ with the second coordinate equal to~$d$
is~{\em upper}.
\smallskip

Now we are ready to formulate the main result.
\begin{theorem}\label{th.10-1}
 For every~$u \in \cU_{\, \Omega}$, we have;
\begin{equation}\label{eq.10-1}
wM \ \le \ \|u_x\|_1 \ \le \ 2wM\, ,
\end{equation}
where~$\max\limits_{(x,y)\in \Omega}u(x,y) = M$.
\end{theorem}

\noindent {\tt Proof}. Since for every $y\in [c,d]$, we have $m(y) \le M$,
Proposition~1 implies that $\|u_x\|_1 \le \int_{c}^d\, M \, dx \, = \, 2(d-c)M$, which proves the upper bound.  To establish the lower bound, we apply
the concavity of~$u$: let $\bz \in \Om$ be a point of maximum of $u$; then~$\ u\bigl((1-t)\bc + t\bz\bigr) \, \ge \,
(1-t)u(\bc) \,  + \, t\, u(\bz)\, = \, tM$ for every~$t\in [0,1]$.
Hence, the maximum of~$u$ on the horizontal line passing through the point
$(1-t)\bc + t\bz$ is at least~$tM$.
The same is true for every point~$(1-t)\bd + t\bz$ of the segment~$[\bz, \bd]$.
Now Proposition~1 yields~$\|u_x\|_1 \, = \, \int_{c}^d \, 2m(x)\, dx\, = \,
\int_{c}^z \, 2m(y)\, dy\, + \, \int_{z}^d \, 2m(y)\, dx\, \ge \,
2(z-c)\int_{0}^1 \, tM\, dt\, + \, 2(d-z)\int_{0}^1 \, tM\, dt\,
= \,  (d-c)M$.

{\hfill $\Box$}
\smallskip

Both bounds for the value~$\|u_x\|_1$ in~(\ref{eq.10-1}) are sharp as the following
theorem states:
\begin{theorem}\label{th.15-1}
The lower bound in~(\ref{eq.10-1}) is attained if and only if there exists a point of maximum~$\bz \in \Omega$ of~$u$  and lower and upper points~$\bc$ and~$\bd$ of~$\Omega$ respectively
such that the function $u$ is linear on each of the two segments~$[\bz, \bc]$ and $[\bz, \bd]$
and every point of those segments is a point of maximum for the restriction of~$u$
to the horizontal  line passing through it.

The upper bound can be attained only in distributions.
It occurs if and only if there are
lower and  upper points of~$\Omega$
such that an open segment between them consists of points of maximum of~$u$.
\end{theorem}

\smallskip

\noindent {\tt Proof}.
If~$\|u_x\|_1 =  2wM$, then every horizontal cross-section of~$\Omega$
intersects the set~$S$ of points of maxima of~$u$,  which is impossible
for cross-sections close to the lower and upper ones.
If~$u$ is discontinuous, then
this equality holds in distributions
precisely when all horizontal cross sections, except of the extreme ones, intersect~$S$.
Since~$S$ is convex, this means that it contains an open interval connecting some lower and upper points of~$\Omega$.

From the proof of the lower bound in~(\ref{eq.10-1}) it
follows  that~$\|u_x\|_1 =  wM$ if and only
if~$\ u\bigl((1-t)\bc + t\bz\bigr) \, = \,
(1-t)u(\bc) \,  + \, t\, u(\bz)\, = \, tM$ for every~$t\in [0,1]$, i.e.,~$u$ is linear  on the segment~$[\bz, \bc]$.
Furthermore,
the maximum of~$u$ on the horizontal line passing through the point
$(1-t)\bc + t\bz$ must be equal to~$tM$, i.e., must be achieved
at this point. This completes the proof for the segment~$[\bz, \bc]$.
For the segment~$[\bz, \bd]$,  the proof is the same.

{\hfill $\Box$}
\smallskip

\begin{remark}\label{r.10-1}
{\em Theorem~1 implies that for every convex domain~$\Omega$,  the lower
bound~in~(\ref{eq.10-1}) is achieved on a certain admissible function~$u$
and the upper bound is achieved in distributions on another function.
A question arises can they be achieved on the same function but for different
directional derivatives?  To attack this problem we need to introduce some further notation. }
\end{remark}

The intersection of $\Omega$ with its line of support is called an {\em edge} of~$\Omega$.
One-point edges a precisely the extreme points of~$\Omega$. We say that a one-point edge
is {\em parallel} to a given line if it is an intersection of~$\Omega$ with a line of support parallel to this line.
For an arbitrary unit vector~$\bh$, we consider the
directional derivative~$u_{\bh} = (\nabla u, \bh)$.
We call {\em the width of $\Omega$ parallel to~$\bh$}
a pair of lines of support parallel to~$\bh$ and also the distance between them.

\smallskip

\begin{cor}\label{c.10-1}
Let~$\bh_1, \bh_2$ be arbitrary unit vectors and
$w_i$ be the width parallel to~$\bh_i$;  then
\begin{equation}\label{eq.20-1}
\frac{\|u_{\bh_1}\|_1}{\|u_{\bh_2}\|_1}\quad \le \quad 2\ \frac{w_1}{w_2} \, .
\end{equation}
The equality can be achieved only in the sense of distributions.
This occurs precisely when~$u$ is constant on a line
connecting two points~$\bc, \bd$ on the opposite edges parallel to~$\bh_1$,
 linear on each half-plane  about the line~$\bc\bd$ and vanishing
 on the lines of support parallel to~$\bh_2$. In this case~$\bh_2$ is parallel to~$\bc\bd$.
\end{cor}

From Corollary \ref{c.10-1} follows the following result generalizing the estimate \eqref{estimateM} for Minimax problem \ref{pr.minimax},
$$
\frac{\sup_{u} \frac{\|u_{\bh_1}\|_1}{\|u_{\bh_2}\|_1}}{\inf_{u} \frac{\|u_{\bh_1}\|_1}{\|u_{\bh_2}\|_1}}
\ =\ \sup_{u} \frac{\|u_{\bh_1}\|_1}{\|u_{\bh_2}\|_1}\ \sup_{u} \frac{\|u_{\bh_2}\|_1}{\|u_{\bh_1}\|_1}\ \le \
2\ \frac{w_1}{w_2} \ 2\ \frac{w_1}{w_2} \ =\ 4.
$$

\smallskip

\noindent\textbf{Proof of Theorem~\ref{th.5-1}}. We apply Corollary~\ref{c.10-1} to
the unit coordinate vectors~$\be_1, \be_2$.

{\hfill $\Box$}
\smallskip

\smallskip

Corollary~\ref{c.10-1} gives the geometric conditions on the domain~$\Omega$ and directions~$\bh_1, \bh_2$ under which the inequality
\begin{equation}\label{eq.30-1}
\sup_{u \in \cU_{\, \Omega}}\, \frac{\|u_{\bh_1}\|_1}{\|u_{\bh_2}\|_1}\quad  \le \quad  2\ \frac{w_1}{w_2} \, .
\end{equation}
becomes an equality.

\smallskip

\begin{cor}\label{c.20-1}  For a given convex domain~$\Omega$ and
non-collinear unit vectors~$\bh_1, \bh_2$, the inequality~(\ref{eq.30-1}) becomes an equality
if and only if
there is a straight line parallel to~$\bh_2$ intersecting two
edges of~$\Omega$ parallel to~$\bh_1$.
\end{cor}
\smallskip

Now let us turn to the uniform bounds over all directions~$\bh_1, \bh_2$.
Denote the biggest and the smallest widths of~$\Omega$ by $w_{\max}$ and $w_{\min}$
respectively.
Applying~(\ref{eq.20-1}) we obtain

\smallskip
\begin{theorem}\label{th.20-1}  For a given convex domain~$\Omega$, we have
\begin{equation}\label{eq.40-1}
\sup_{u, \bh_1, \bh_2}\, \frac{\|u_{\bh_1}\|_1}{\|u_{\bh_2}\|_1}\quad  \le \quad 2\, \frac{w_{\max}}{w_{\min}} \ .
\end{equation}
 \end{theorem}

 Theorem~\ref{th.20-1} yields the conditions under which the upper bound~(\ref{eq.40-1})
 is sharp.
\smallskip

\begin{theorem}\label{th.30-1}
For a convex domain~$\Omega$, the inequality~(\ref{eq.40-1}) becomes an
equality
 if and only if  there are the greatest and the smallest width
of~$\Omega$
with orthogonal directions. If this is the case, then
$\bh_1, \bh_2$ are parallel to those directions respectively.
\end{theorem}

\noindent \textbf{Proof.}
For each maximal width, the corresponding  opposite  edges possess
 a common perpendicular with ends on~$\Omega$. By Corollary~1, if those lines are parallel
to~$\bh_1$, then~the segment~$\bc\bd$ must be their common perpendicular
parallel to~$\bh_2$.

{\hfill $\Box$}
\smallskip

\begin{cor}\label{c.30-1}
If the inequality~(\ref{eq.40-1}) becomes an equality, then
the directions~$\bh_1$ and~$\bh_2$ are orthogonal and correspond to the
maximal and minimal widths of~$\Omega$.
\end{cor}
\smallskip
\begin{ex}\label{ex.10-1}
{\em  If $\Omega = \bigl\{(x, y) \in \re^2: \
x^2 + y^2 \le 1\bigr\}$ is a disc, then there exists a unique, up to rotations and normalization, function realizing (in distributions)
the biggest ratio~$\frac{\|u_{x}\|_1}{\|u_{y}\|_1} =2$.
This is~$u(x,y) = 1-|x|$.  }
\end{ex}

\smallskip

\begin{ex}\label{ex.20-1} {\em If $\Omega$ is a non-isosceles triangle,
then there exists a unique, up to normalization, function realizing (in distributions)
the biggest ratio~$\frac{\|u_{\bh_1}\|_1}{\|u_{\bh_2}\|_1}\, =\, \frac{2a}{h}$, where
$a$ is the longest side and $h$ is the altitude to it.
This is the linear function equal to one on the longest side and to zero
at the opposite vertex. }
 \end{ex}
\bigskip

Corollary~\ref{c.20-1} and Theorem~\ref{th.30-1} give geometric conditions for the maximal possible
values of the ratio between the $L_1$-norms of the partial derivatives. What
happens if those conditions are not satisfied? We now formulate two open problems.
\smallskip

\begin{prob}\label{pr.10-1}
For a given convex domain~$\Omega$ and
two directions~$\bh_1, \bh_2$ find the supremum of the ratio
$\frac{\|u_{\bh_1}\|_1}{\|u_{\bh_2}\|_1}$ over all admissible functions~$u$.
\end{prob}
\smallskip

Corollary~\ref{c.10-1} states only that this supremum does not exceed~$\frac{2w_1}{w_2}$.

\smallskip

\begin{prob}\label{pr.20-1}
For a given convex domain~$\Omega$,
find the supremum of the ratio
$\frac{\|u_{\bh_1}\|_1}{\|u_{\bh_2}\|_1}$ over all admissible functions~$u$ and directions
$\bh_1, \bh_2$.
\end{prob}

Theorem~\ref{th.20-1} limits this  supremum by the value~$\ 2\, \frac{w_{\max}}{w_{\min}}$.

\vspace{1cm}

\begin{center}
\large{\textbf{3. The case~$\ \bp \  \mathbf{\in \, (1, +\infty)}$}}
\end{center}
\medskip

If~$1< p< \infty$, then the supremum of the ratio between~$\|u_x\|_{p}$ and~$\|u_y\|_{p}$
is finite for some domains~$\Omega$ and infinite for others. The main result of this section, Theorem~\ref{th.10-p}, gives a criterion for deciding between those two cases.
To formulate it we need to introduce some notation.
The {\it tangent cone} to $\Om$ with the vertex at $\xi \in \pl\Om$ is the intersection of all closed half-planes containing $\Om$ and bounded by lines of support through $\xi$.
A point $\xi \in \partial\Omega$ is called {\it regular}, if there is only one line of support to $\Omega$ through~$\xi$, and {\it singular} otherwise. Equivalently, $\xi$ is regular, if the tangent cone  to $\Omega$ with the vertex at $\xi$ is a half-plane, and singular otherwise.

Consider a tangent cone to $\Omega$ with the vertex at a singular point of $\partial\Omega$.
Any straight line that intersects the tangent cone by its vertex is called an {\it angular} support line.

\begin{figure}[h]
\begin{picture}(0,140)
 \scalebox{1}{

 \rput(1.3,2.4){
 \psline[linestyle=dashed](0,-2)(0,2) \psline[linestyle=dashed](5,-2)(5,2)
 \pscurve(0,0)(1.5,0.9)(3.5,1)(5,0.5) \pscurve(0,0)(1.5,-1)(4,-0.6)(5,0.5)
\rput(2.5,0.2){\scalebox{1.5}{$\Om$}} \rput(-0.8,-1.7){(a)}
}

\rput(10,2.4){
\psellipse(2.5,0)(2.0,2.4)
\pspolygon[fillstyle=solid,fillcolor=white,linewidth=0pt,linecolor=white](1.77,0)(5,0)(5,-2.5)(1.77,-2.5)
\psline(4.5,0)(1.77,-2.23)
\psline[linestyle=dashed](0.5,-2)(0.5,2) \psline[linestyle=dashed](4.5,-2)(4.5,2)
\rput(2.5,0.2){\scalebox{1.5}{$\Om$}} \rput(-0.8,-1.7){(b)}
}

}
\end{picture}
\caption{In figure (a), both vertical lines of support to $\Omega$ are angular, while in figure (b), both lines are not angular: the left line is tangent, and the right one is half-tangent to $\Omega$.}
\label{fig:ang}
\end{figure}
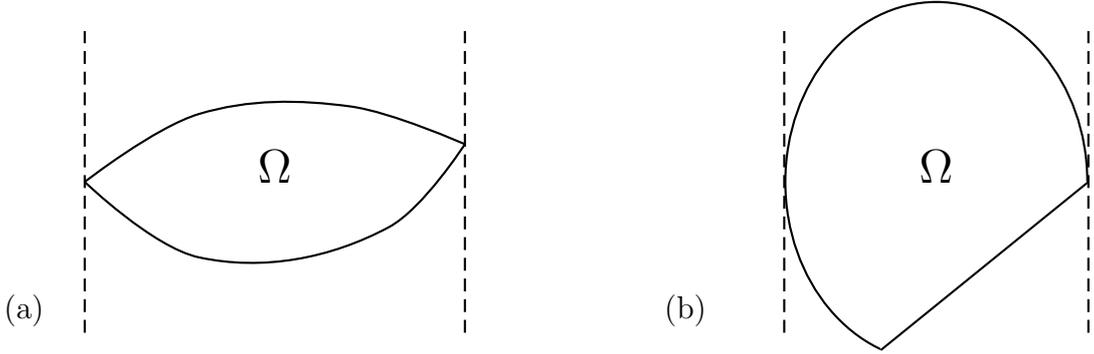

\medskip

\begin{theorem}\label{th.10-p} For a domain~$\Omega$, we have~$K_p < \infty$
if and only if both vertical (that is, parallel to the $y$-axis) lines of support
of~$\Omega$  are angular.
\end{theorem}
\noindent {\tt Proof}. {\em (Necessity. If at least one vertical line of support is not angular,
then~$K_p = \infty$). }
\smallskip

Denote by $l_\al$ the line of support with the outward normal $(\cos\al, \sin\al)$,\, $-\pi < \al \le \pi$, and by $\Pi_\al$ the closed half-plane bounded by $l_\al$ and containing $\Om$. By the hypothesis, one of the lines $l_0$, $l_\pi$ is not angular; let it be $l_0$. Fix a value $0 < \vphi < \pi/2$ and set
$$
\Om' = \Om'_\vphi = \cap_{|\al|\ge\vphi} \Pi_\al;
$$
see Fig.~\ref{fig:AB}. Since $l_0$ is not angular, there is an open arc of the boundary $\pl\Om$ contained in the interior of $\Om'$. All lines of support at points of this arc correspond to angles in $(-\vphi,\, \vphi)$. Choose a point $\xi$ on the arc, and denote $r = \text{dist}(\xi, \pl\Om') > 0$.

Take $\ve > 0$ and denote
$$
D = D(\vphi,\ve) := \{ (x,y) :\ \text{dist}\big( (x,y), \xi \big) \le r/2, \ \, \text{dist}\big( (x,y), \pl\Om \big) \ge \ve \} \cap \Om
$$
(see Fig.~\ref{fig:AB}). Let $u = u_{\vphi,\ve}$ be the smallest concave function satisfying $u\rfloor_{\pl\Om} = 0$ and $u\rfloor_{D} = 1$. Each point of the graph of $u$ belongs

(i) either to the relative interior point of $D \times \{ 1 \}$;

(ii) or to a plane through a point of $\pl D \times \{ 1 \}$ and a point of $(\pl\Om \setminus \pl\Om') \times \{ 0 \}$;

(iii) or to a plane through a point of $\pl D \times \{ 1 \}$ and a point of $(\pl\Om \cap \pl\Om') \times \{ 0 \}$.\\
Correspondingly, $\Om$ is the disjoint union of three sets, $\Om = \Om_0 \cup \Om_1 \cup \Om_2$, where $\Om_0$ is the interior of $D$, and $\nabla u\rfloor_{\Om_0} = 0$;\, $\arg(\nabla u) \not\in [-\vphi,\, \vphi]$ at each regular point of $\Om_1$;\, and $\arg(\nabla u) \in [-\vphi,\, \vphi]$ at each regular point of $\Om_2$.

\begin{figure}[h]
\begin{picture}(0,160)
\scalebox{1}{

\rput(7.5,0.1){
\psarc(0,0){5}{0}{180} \psline(-5,0)(5,0) \psline[linestyle=dashed](2.5,4.33)(7.5,1.443)(5,0)
     \psdots(4.787,1.443)   
\pspolygon[linewidth=0pt,linecolor=white,fillcolor=lightgray,fillstyle=solid] (4.45,0.83)(4.23,1.0)(4.09,1.4)(4.18,1.8)(4.45,2.025)(4.66,1.443)(4.82,0.75)
\psarc[linewidth=0.8pt](4.787,1.443){0.7}{120}{275}   \psarc[linewidth=0.8pt](0,0){4.9}{8.5}{24.75}

\rput(5.1,1.443){\scalebox{1}{$\xi$}}  \rput(4.4,1.4){\scalebox{1}{$D$}}  
}}
\end{picture}
\caption{Here $\Om$ is the upper half-circle, and $\Om'$ contains $\Om$ and is bounded by the dashed broken line on the right hand side.}
\label{fig:AB}
\end{figure}
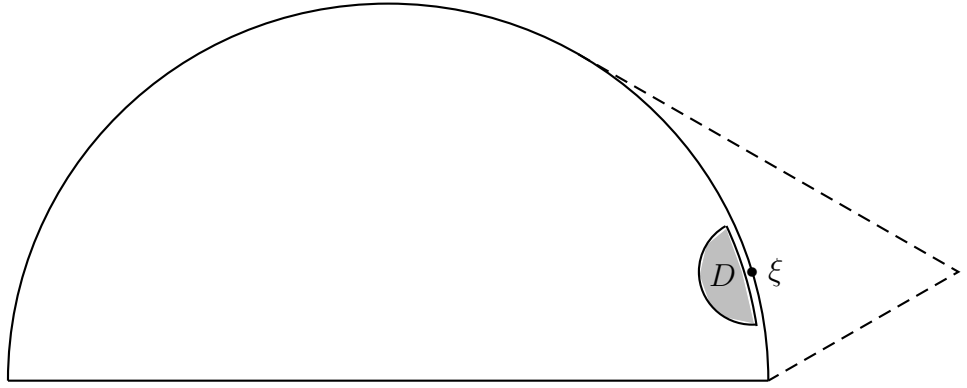

Thus, one has
$$
\int_\Om |\nabla u|^p\, dx dy = I_1 + I_2,
$$
where
$$
I_j  = I_j(\vphi,\ve) = \int_{\Om_j} |\nabla u|^p\, dx dy, \quad j = 1,\, 2.
$$
For regular points of $\Om_1$ one has $|\nabla u| \le \dfrac{1}{r/2}$, hence
$$
I_1 \le |\Om| (2/r)^p.
$$
Here and in what follows, $|\Om|$ means the area of $\Om$.

Further, the set $\Om_2$ contains the layer $L$ of width $\ve$ between $D$ and $\pl\Om$. The area of this layer is of the order of $\ve$, and $|\nabla u| = 1/\ve$ at all points of the layer, therefore
$$
I_2 \ge \int_{L} |\nabla u|^p\, dx dy \sim \frac{\ve}{\ve^{p}} \to \infty \quad \text{as} \quad \ve \to 0.
$$

On the other hand, at each regular point of $\Om_2$ one has $|u_y| \le \tan\vphi\, |u_x|$, hence
$$
|\nabla u|^p = \big( u_x^2 + u_y^2 \big)^{p/2} \le (1 + \tan^2 \vphi)^{p/2} u_x^p = \frac{1}{\cos^p \vphi}\, u_x^p,
$$
and
$$
\int_{\Om_2} u_x^p\, dx dy \ge \cos^p \vphi \int_{\Om_2} |\nabla u|^p\, dx dy = \cos^p \vphi\, I_2
\quad \Longrightarrow \quad \int_{\Om_2} u_y^2\, dx dy \le (1 - \cos^p \vphi )\, I_2.
$$

Fix $\vphi$ (and therefore, $r$) and let $\ve \to 0$; then $I_2$ goes to infinity, and $I_1$ is bounded, and therefore, the integrals
$$
I_{1x} := \int_{\Om_1} u_x^p\, dx dy \quad \text{and} \quad I_{1y} := \int_{\Om_1} u_y^p\, dx dy
$$
are also bounded. It follows that
$$
\frac{\int_\Om u_x^p\, dx\, dy}{\int_\Om u_y^p\, dx\, dy} =
\frac{\int_{\Om_1} u_x^p\, dx\, dy + \int_{\Om_2} u_x^p\, dx\, dy}{\int_{\Om_1} u_y^2\, dx\, dy + \int_{\Om_2} u_y^2\, dx\, dy} \ge
\frac{I_{1x} + \cos^p \vphi\, I_2}{I_{1y} + (1 - \cos^p \vphi)\, I_2},
$$
hence the lower partial limit of the ratio
$$
\liminf_{\ve\to\infty}\, \frac{\int_\Om u_x^p\, dx\, dy}{\int_\Om u_y^p\, dx\, dy} \ge \frac{\cos^p \vphi}{1 - \cos^p \vphi}.
$$
Since $\vphi$ can be made arbitrarily small, the limit of the ratio is $+\infty$. This proves the part of necessity of Theorem \ref{th.10-p}.
\bigskip

 \noindent {\em (Sufficiency. If both vertical lines of support are angular,
then~$K_p < \infty$). }
\smallskip

Without loss of generality it may be assumed
that~$\Omega$ has the extreme left point at the origin~$(0,0)$.
Its extreme right point
 has coordinates~$(w_x, c)$, where~$w_x> 0$ is the length of the projection
 of~$\Omega$ to the~$x$-axis.

After the linear change of variables
\begin{equation}\label{eq.affine-p}
x \ = \ \frac{w_x}{2}\, \tilde x\, , \qquad
y \ =  \ \frac{c}{2}\, \tilde x \ + \ \tilde y
\end{equation}
the domain~$(x, y) \in \Omega$
is mapped to a domain~$(\tilde x, \tilde y) \in \tilde \Omega $  with the extreme
left point~$(0,0)$ and extreme right point~$(2,0)$. The Jacobian of this transform is equal to~$\frac{w_x}{2}$
and~$u_{y}\, = \, u_{\tilde y}, \ u_{x}\, = \,
\frac{2}{w_x}\, u_{\tilde x} \, -  \, \frac{c}{w_x}\, u_{\tilde y}$. Therefore,
$$
\bigl\|u_y \bigr\|_p\ = \ \bigl(w_x/2 \bigr)^{1/p}\bigl\|u_{\tilde y} \bigr\|_p\, ;
\qquad \bigl\|u_x \bigr\|_p\ \le \ \bigl(w_x/2 \bigr)^{1/p}
\left(\frac{2}{w_x}\, \bigl\|u_{\tilde x} \bigr\|_p\ + \
\frac{|\,c\,|}{w_x}\, \bigl\|u_{\tilde y} \bigr\|_p\right)\, ,
$$
where the $L_p$-norms of~$u_{\tilde x} , u_{\tilde y}$ are computed over the
domain~$\tilde \Omega$. Thus,
\begin{equation}\label{eq.change-p}
 \frac{\|u_x \|_p}{\|u_y \|_p}\quad \le \quad
\frac{2}{w_x}\,\frac{\|u_{\tilde x}\|_p}{\|u_{\tilde y}\|_p}\ + \ \frac{|\,c\,|}{w_x}\  .
 \end{equation}
  Thus,
it suffices to prove the theorem for the domain~$\tilde \Omega$.
To simplify the notation
we will  assume that the original domain~$\Omega$ has the extreme left and
right angular points~$(0,0), \, (2,0)$,
so we shall not use tildes. This way we obtain an upper bound
for~$\frac{\|u_{\tilde x}\|_p}{\|u_{\tilde y}\|_p}$ and then substitute it to~(\ref{eq.change-p}).

Let~$\Omega$ be bounded below and above by the graphs of functions~$\varphi_1(x)$
and~$\varphi_2(x)$ respectively. So, $\varphi_1$ is convex,~$\varphi_2$
is concave on the segment~$[0,2]$ and both those functions
 vanish at the ends of the segment.

 We estimate the ratio of $L_p$-norms of~$u_x$ and~$u_y$ on the
left part~$\frac12 \, \Omega \, = \, \bigl\{ (x, y) \in \Omega: \, 0\le x\le 1 \bigr\}$, then the same argument can be applied to the right part. In other words, we
prove that $\|u_x\|_{L_p(\frac12 \Omega)} \le M_p \|u_y\|_{L_p(\frac12 \Omega)}$
for some constant~$M_p$.
Then the same argument gives the same
estimate for the right part of~$\Omega$. Summing two parts, we get
$\|u_x\|_{L_p(\Omega)} \le M_p \|u_y\|_{L_p(\Omega)}$, and hence~$K_p \le M_p$.
We will use the short notation~$\|u\| = \|u\|_{L_p(\frac12 \Omega)}$, and
similarly for~$u_x, u_y$. We begin the the following key lemma.

\begin{lemma}\label{l.10-p} At every point
$(x,y)  \in  \frac12 \, \Omega$ where~$u$ is differentiable, we have $\, \bigl|u_x\bigl| \, \le \,
\frac1x\, \bigl(\, u\, - \, y\, u_y \bigr)$.
\end{lemma}

\noindent {\tt Proof}. Denote by~$L(x,y)$ the tangent plane to the
graph of the function~$u$ at the point $(x,y)\in \frac12 \, \Omega$; see Fig.~\ref{fig:plane}. The equation of this
plane in~$\re^3$ is
$$
(x', y', z') \in L(x,y) \qquad \Leftrightarrow \qquad \ z'-z \ = \ (x'-x)\, u_x(x,y) \ + \ (y'-y)\, u_y(x,y)\, ,
$$
where $z\, =\, u(x,y)$.
 Since $u$ is concave, it follows that
 $L(x,y)$ is located  above
the graph. i.e., $z' \, \ge \, u(x',y')$.
At the point~$(0,0)$, we have~$u = 0$ and therefore,
$z'\ge  0$.  Substituting the point~$(0,0,z')$ to the equation of the plane, we get
$z + (-x )u_x + (-y)u_y \, = \, z' \ge 0$. Thus,
$\, u  \, -\, x \, u_x \, - \, y\, u_y \, \ge \, 0$ and so,
$u_x \, \le \,
\frac1x\, \bigl(\, u\, - \, y\, u_y \bigr)$.\
 This proves the lemma in the case~$u_x(x,y) \ge 0$.

If $u_x(x,y) < 0$, then we
apply the same argument to the
right angular point~$(2,0)$, where  we also have~$z' \ge 0$.
Substituting the point~$(2,0,z')$ to the equation of~$L(x,y)$, we obtain
 $z + (2-x )u_x + (-y)u_y \, \ge 0$,
which implies $u_x \, \ge \,
\frac{1}{2-x}\, \bigl(\, - u\, + \, y\, u_y \bigr)$.
Consequently, $|u_x| \, = \, -u_x \, \le \, \frac{1}{2-x}\,
\bigl(\, u\, - \, y\, u_y \bigr)$. The right hand side must be positive,
and  since~$x\le 1$, we have
$2-x \ge x$, which implies  $|u_x| \,  \le \, \frac{1}{x}\,
\bigl(\, u\, - \, y\, u_y \bigr)$. This completes the proof of the lemma.

\begin{figure}[h]
\begin{picture}(0,200)
\scalebox{0.85}{

\rput(5.5,1.5){

\pscustom[fillstyle=vlines,fillcolor=yellow]{
\pscurve(0,0)(1.5,-1)(4.4,-1.5)(6.7,-1)(8,0)
\pscurve[linestyle=dashed](8,0)(6.5,1)(4,1.35)(1.35,0.9)(0,0)
}

\pspolygon[linewidth=0pt,linecolor=white,fillstyle=solid,fillcolor=white](2.65,0)(3.6,0)(3.6,0.45)(2.65,0.45)

\pspolygon[linewidth=0pt,linecolor=white,fillstyle=solid,fillcolor=white](5.8,-0.27)(5.8,0.28)(6.4,0.28)(6.4,-0.27)

\psdots[dotsize=3pt](2.5,0.25)(2.5,3.25) (0,1.8)(8,4.8)
\pspolygon(-1.5,0)(-4.5,2.5)(7.5,7)(10.5,4.5) \pscurve(0,0)(1,2.7)(4.5,4.3)(7.25,2.75)(8,0)
\psline[linestyle=dashed](2.5,0.25)(2.5,3.25)
     \psline(0,0)(0,0.4)       \psline[linestyle=dashed](0,0.4)(0,1.8)
     \rput(-0.8,1.85){\scalebox{0.9}{$(0,0,z')$}}     \rput(8.7,5.15){\scalebox{0.9}{$(2,0,z')$}}
         \psline(8,0)(8,3.4)       \psline[linestyle=dashed](8,3.4)(8,4.8)
\rput(6.1,0){\scalebox{1.5}{$\Om$}} \rput(3.1,0.25){\scalebox{0.9}{$(x,y)$}} \rput(8.55,0){\scalebox{1}{$(2,0)$}}  \rput(-0.55,-0.05){\scalebox{1}{$(0,0)$}} \rput(3.22,3.1){\scalebox{0.9}{$(x,y,z)$}}  \rput{23}(-0.4,3.5){\scalebox{1.25}{$L(x,y)$}}
\rput(4.55,3.82){\scalebox{1.2}{$u(x,y)$}}
\rput{-27}(1.1,-1.2){\scalebox{1.17}{$y = \vphi_1(x)$}}    \rput{-16}(6.4,1.5){\scalebox{1.17}{$y = \vphi_2(x)$}}
}}
\end{picture}
\caption{The tangent plane to the graph of $u$.}
\label{fig:plane}
\end{figure}
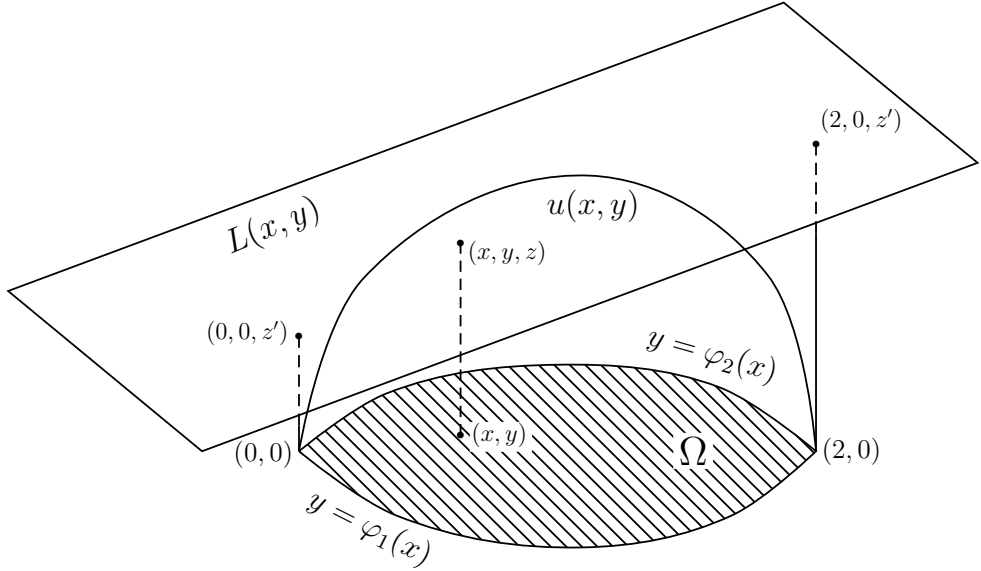

{\hfill $\Box$}
\smallskip

 Now we continue the proof of Theorem~\ref{th.10-p}. Lemma~\ref{l.10-p} yield that
\begin{equation}\label{eq.both-p}
\bigl\|u_x \bigr\| \ \le \ \Bigl\|\frac1x \, u \Bigr\|\ + \
\Bigl\|\frac{y}{x} \, u_y \Bigr\| \, ,
\end{equation}
 where we use the short notation for the functions
 $\frac1x \, u \, = \, \frac1x \, u(x,y)\, , \
 \frac{y}{x} \, u \, = \, \frac{y}{x} \, u(x,y)$. We estimate the two terms in the
 right hand side of~(\ref{eq.both-p}) separately, beginning with the the second one.

 The convex domain~$\Omega$ lies
between two its tangents to~$\Omega$ at the point~$(0,0)$.
Consequently, for every~$(x,y) \in \Omega$, we have $\varphi_1'(0) \le \frac{y}{x} \le \varphi_2'(0)$.
In the notation~$m = \max \{|\varphi_1'(0)|\, , \, |\varphi_2'(0)|\}$, we get
$\bigl|\frac{y}{x}\bigr| \, \le \, m$ and therefore,
\begin{equation}\label{eq.first-p}
\Bigl\|\, \frac{y}{x} \, u_y \, \Bigr\|\ \le \ m\, \|u_y\|\, .
\end{equation}
 To estimate the
first term~$\bigl\|\frac{1}{x} \, u \bigr\|$ we use the generalized Poincar\'e  inequality
 (see, e.g., \cite{FGR10, GGR18}): for an absolutely continuous function~$v$
 on~$[0,1]$ with derivative from~$L_p$ and with~$v(0) = v(1) = 0$, we have
\begin{equation}\label{eq.Poinc-p}
\int_{0}^1  |v(t)|^p \, dt \quad  \le \quad
C_p\, \int_{0}^1  |\dot v(t)|^p\, dt\, ,
\end{equation}
where $C_p$ is a fixed constant. The precise value of this constant is found in~\cite[Theorem 2.2]{FGR10}.
After the change of variables~$y= (1-t)\varphi_1(x) \, + \, t\varphi_2(x)$, we obtain
$$
\int_{\varphi_1(x)}^{\varphi_2(x)}  |v(y)|^p \, dy \quad \le \quad
C_p\, \left|  \varphi_2(x) - \varphi_1(x) \right|^p
\int_{\varphi_1(x)}^{\varphi_2(x)}  |\dot v(y)|^p \, dy\, .
$$
Since the function~$\varphi_2(x) - \varphi_1(x)$ is  concave in~$x$,   it follows that
$$
\varphi_2(x) \ - \ \varphi_1(x)\ \le \
 \bigl(\varphi_2'(0) - \varphi_1'(0)\bigr)\, x\, \ \le \ 2\,m\,x.
$$
Thus,
$$
\int_{\varphi_1(x)}^{\varphi_2(x)}  |v(y)|^p \, dy \quad \le \quad
C_p\, (2mx)^p
\int_{\varphi_1(x)}^{\varphi_2(x)}  | \dot v(y)|^p \, dy\, .
$$
We have
$$
\Bigl\| \, \frac1x \, u \, \Bigr\|^p \ = \
\int_{0}^{1}\, \frac{dx}{x^p}\, \int_{\varphi_1(x)}^{\varphi_2(x)} \,
u^p\, \, dy \quad \le \quad  \int_{0}^{1}\,
\,   \frac{dx}{x^p} \ C_p (2mx)^p \,
  \, \int_{\varphi_1(x)}^{\varphi_2(x)} \,
|u_y|^p\, \, dy \ = \
$$
$$
= \ C_p 2^{\,p} m^p\, \int_0^1 dx \, \int_{\varphi_1(x)}^{\varphi_2(x)} \,
|u_y|^p\, \, dy\quad  = \quad  C_p 2^{\,p} m^p \, \bigl\| \, u_y \, \bigr\|^p\, .
$$
Thus, $\bigl\| \, \frac1x \, u \, \bigr\| \, \le \,
2\, (C_p)^{1/p}\, m\,  \bigl\| \, u_y \, \bigr\|$.
Combining this with inequality~(\ref{eq.first-p}) and
substituting to~(\ref{eq.both-p}), we finally obtain
\begin{equation}\label{eq.final-p}
\bigl\| u_x \bigr\| \quad \le \quad
m \, \bigl( 2\, (C_p)^{1/p}\ + \ 1\bigr) \, \bigl\| u_y \bigr\|\, .
\end{equation}
Thus, if the domain~$\Omega$ has the left and the right
angular points~$(0,0)$ and~$(2,0)$ respectively, then
for every~$u\in \cU_{\,\Omega}$,
we have~$\frac{\|u_x\|_p}{\|u_y\|_p}\, \le \, M_p \, \, = \, m \, ( 2\, (C_p)^{1/p}\ + \ 1)$.
Therefore, $K_p(\Omega) \ \le \ M_p$, which completes the proof.

{\hfill $\Box$}
\smallskip

Equality~(\ref{eq.final-p}) estimates~$K_p(\Omega)$ for the domains
that have angular points~$(0,0)$ and~$(2,0)$.
For arbitrary domains~$\Omega$,  we apply formulas~(\ref{eq.change-p}) and obtain
\begin{theorem}\label{th.20-p}
If the domain~$\Omega$ has the left and right angular points~$(0,0)$ and~$(2,0)$,
then
$$
K_p(\Omega)\ \le \ M_p(\Omega)\ = \ m \, ( 2\, (C_p)^{1/p}\ + \ 1)\, ,
$$
where~$C_p$ is the constant in the generalized Poincar\'e inequality~(\ref{eq.Poinc-p}),
$m$ is the largest modulus of the four angle coefficients
of tangents at the  angular points~$(0,0)$ and~$(2,0)$.

For an arbitrary domain~$\Omega$, we have
$$
K_p(\Omega) \ \le \
\frac{2}{w_x}\, M_p(\tilde \Omega)\ + \  \frac{|\,c\,|}{w_x},
$$
where~$w_x$ is the length of the  projection of~$\Omega$
to the~$x$-axis,~$c$ is the difference of~$y$-coordinates of the left and right angular points,
$\tilde \Omega$ is the domain obtained from~$\Omega$ by the affine
transform~(\ref{eq.affine-p}).

\end{theorem}

\begin{prob}\label{pr.10-p}
Find~$K_{p}(\Omega)$ for an arbitrary domain~$\Omega$.
\end{prob}
A natural question arises for which domain the constant~$K_p(\Omega)$ is minimal?
Since a linear contraction of~$\Omega$ with respect to the $x$-axis
enlarges~$K_p$, the problem has to be formulated after a normalization of the sizes of~$\Omega$
along the coordinate axes, for example, by the condition that~$\Omega$ has equal projections.
\begin{prob}\label{pr.12-p}
Which domain inscribed in a square with sides parallel to the coordinate axes has the minimal
constant~$K_{p}(\Omega)$?
\end{prob}
Note that for~$p>1$, neither the  circle nor other domains with smooth boundaries give the answer to Problem~\ref{pr.12-p}, since
by Theorem~\ref{th.10-p}, for such domains,~$K_{p} = \infty$.   On the other hand,
for~$p=1$, the circle is optimal (Example~\ref{ex.10-1}), however, there are many other optimal figures.

\bigskip

\begin{center}
\large{\textbf{4. The case~$\ \bp \ \mathbf{= \ \infty}$}}
\end{center}
\medskip

This remaining case admits an explicit expression for~$K_{p}(\Omega)$. A simple formula for~$K_{\infty}(\Omega)$
 is provided by Theorem~\ref{th.p=infty} below. Note that for all~$p \in [1, +\infty)$, we do not have formulas but  only estimates for~$K_p$.

Let $\Om$ be bounded below and above by the graphs of functions $\vphi_1(x)$ and $\vphi_2(x)$, $x \in [a,\, b]$, respectively. Thus, $\vphi_1$ is convex and $\vphi_2$ is concave.
Thus,  $\Om = \{ (x,y) : x \in [a,\, b],\ \vphi_1(x) \le y \le \vphi_2(x) \}$. Denote as above  $m = \max \{ |\vphi_i'(x)| :\, i \in \{1,2\},\ x \in [a,\, b] \}.$ Clearly, the maximum is attained at one of the points $x = a$ or $x=b$. If both vertical lines of support are angular then $m < \infty$ and $\vphi_1(a) = \vphi_2(a),\ \vphi_1(b) = \vphi_2(b)$. If, otherwise, one of the vertical lines of support is not angular, then either $m = \infty$ or one of the differences $\vphi_2(a) -\vphi_1(a),\ \vphi_2(b) - \vphi_1(b)$ is positive.

If the intersection of $\Om$ with the line $y=y_0$ is a non-degenerate interval and $(x_0,y_0)$ is one of its endpoints, we let the derivative $u_x$ at $(x_0,y_0)$ to be the one-sided derivative $\frac{d}{dx}\Big\rfloor_{x=x_0} u(x,y_0)$. In the similar way we define the derivative $u_y$ at the boundary points $(x,\vphi_i(x)), \ i \in \{1,2\},\ x \in (a,\, b)$.

\begin{theorem}\label{th.p=infty}
We have $K_\infty = m < \infty$, if both vertical lines of support are angular, and $K_\infty = \infty$ otherwise.
\end{theorem}

Note that it may happen that one or two vertical lines of support are not angular, while $m < \infty$; see Fig.~\ref{fig:lim}\,(b). In this case, obviously, $K_\infty = \infty \ne m$.

\noindent {\tt Proof} is based on the following two lemmas.

\begin{lemma}\label{l.tan}
Assume that $a < x_0 < b$ and $y_0 = \vphi_i(x_0)$, $i \in \{ 1,2 \}$. In other words, the point $(x_0,y_0)$ lies on the graph of the function $\vphi_i$, but does not coincide with one of its endpoints.  Then
$$
-\frac{u_x(x_0,y_0)}{u_y(x_0,y_0)} \ \in \ [ \vphi_i'(x_0-0), \ \vphi_i'(x_0+0)].
$$
\end{lemma}

\noindent {\tt Proof}.
The equation $z = u_x(x_0,y_0) (x-x_0) + u_y(x_0,y_0) (y-y_0)$  determines a plane of support to the graph of $u$ at $(x_0,y_0,0)$. The line of intersection of this plane with the horizontal plane $z=0$ is given by $u_x(x_0,y_0) (x-x_0) + u_y(x_0,y_0) (y-y_0) = 0$, consequently
$$
-\frac{u_x(x_0,y_0)}{u_y(x_0,y_0)}\ =\ \frac{y-y_0}{x-x_0}.
$$
This line lies outside the domain $\Om$, hence it is supporting to the graph of $\vphi_i$ in the horizontal plane, and so, $\frac{y-y_0}{x-x_0}  \in [ \vphi_i'(x_0-0), \ \vphi_i'(x_0+0)]$.

{\hfill $\Box$}
\smallskip

\begin{lemma}\label{l.ess}
The value $\| u_x \|_\infty \, = \, {\rm ess \,sup}_{(x,y)\in\Om} |u_x(x,y)|$ is attained on the boundary of $\Om$, that is,
$$
\| u_x \|_\infty \ = \ \sup_{(x,y)\in\pl\Om} |u_x(x,y)|.
$$
The same formula is true when $u_x$ is substituted with $u_y$.

Assume that $\vphi_1(a) = \vphi_2(a)$ and $\vphi_1(b) = \vphi_2(b)$, and denote by $A$ and $B$ the left and right extreme points of $\Om$, that is, $A = (a,\vphi_1(a))$ and $B = (b, \vphi_1(b))$; then $\| u_x \|_\infty$ is attained on $\Om \setminus \{ A \cup B \}$, that is,
$$
\| u_x \|_\infty \ = \ \sup_{(x,y)\in\pl\Om \setminus \{ A \cup B \}} |u_x(x,y)|.
$$
\end{lemma}

\begin{proof}
Let us first show that
\beq\label{esssup}
{\rm ess \,sup}_{(x,y)\in\Om} |u_x(x,y)| \ = \ \sup_{(x,y)\in\Om'} |u_x(x,y)|,
\eeq
where $\Om'$ is obtained by taking away the horizontal lines through $A$ and $B$ from $\Om$, if $\vphi_1(a) = \vphi_2(a)$ and $\vphi_1(b) = \vphi_2(b)$, and $\Om' = \Om$ otherwise.
It suffices to arrange, for all $\ve > 0$ and $(x,y) \in \Om'$, a set of points in $\Om$ with positive measure for which $u_x$ lies in the $\ve$-neighborhood of $u_x(x,y)$.

Indeed,  there exists a regular point $(x', y')$ arbitrarily close to $(x, y)$ with $u_x(x',y')$ arbitrarily close to $u_x(x,y)$.
In turn, there is a neighborhood of $(x', y')$ such that in all its regular points
$u_x$ is sufficiently close to $u_x(x',y')$.
Formula \eqref{esssup} is proved.

It remains to show that $\sup_{(x,y)\in\Om} |u_x(x,y)|$ is attained on $\pl\Om$, that is,
$$
\sup_{(x,y)\in\Om} |u_x(x,y)| = \sup_{(x,y)\in\pl\Om} |u_x(x,y)|,
$$
and if $\vphi_1(a) = \vphi_2(a)$ and $\vphi_1(b) = \vphi_2(b)$ then
$$
\sup_{(x,y)\in\Om} |u_x(x,y)| = \sup_{(x,y)\in\pl\Om\setminus\{ A\cup B\}} |u_x(x,y)|.
$$
Indeed, take a regular point $(x_0,y_0)$ in $\Om'$; at one of the endpoints of the segment $\Om \cap \{ y=y_0 \}$ (let it be $(x_1,y_0)$) it holds $|u_x(x_1,y_0)| \ge |u_x(x_0,y_0)|$.

This finishes the proof of Lemma \ref{l.ess} for $u_x$. The proof for $u_y$ is similar.
\end{proof}

Let us now prove Theorem \ref{th.p=infty}. First consider the case  when both vertical lines of support to $\Om$ are angular, and therefore,  $m < \infty$. According to Lemma \ref{l.tan}, for any $(x,y)$ on $\pl\Om$ one has
$$
\frac{|u_x(x,y)|}{|u_y(x,y)|} \le m.
$$
Substituting, in place of $(x,y)$, a sequence $(x_i,y_i) \in \pl\Om \setminus \{ A \cup B \}$ such that $|u_x(x_i,y_i)|$ tends to $ \| u_x \|_\infty = {\rm ess \,sup}_{(x,y)\in\Om} |u_x(x,y)|$, and taking into account that $|u_y(x_i,y_i)| \le \| u_y \|_\infty$, one obtains
$$
m \ge \frac{|u_x(x_i,y_i)|}{|u_y(x_i,y_i)|} \ge \frac{|u_x(x_i,y_i)|}{\|u_y\|_\infty} \to \frac{\|u_x\|_\infty}{\|u_y\|_\infty} \quad \text{as} \ \ i \to \infty.
$$

It remains to find functions $u$ with $\frac{\|u_x\|_\infty}{\|u_y\|_\infty}$ arbitrarily close to $m$. Assume without loss of generality that $\vphi_2'(a) = m.$ Fix $\ve > 0$ and choose $\del > 0$ sufficiently small, so as for all $x \in [a,\, a+\del]$, \  $\vphi'_2(x) \in [m-\ve,\, m]$. Take a regular point $(x_0, y_0)$ on the graph of $\vphi_2\rfloor_{[a,a+\del]}$ and denote by $u = u_{\om}$ the smallest concave function satisfying $u\rfloor_{\pl\Om} = 0$ and $u(x_0, y_0-\om) = 1;$ see Fig.~\ref{fig:lim}\,(a). In what follows we assume that $0 < \om \le \om_0$ with $\om_0$ sufficiently small, so as $(x_0, y_0-\om_0)$ is an interior point of $\Om$, and consider the limit $\om \to 0$.

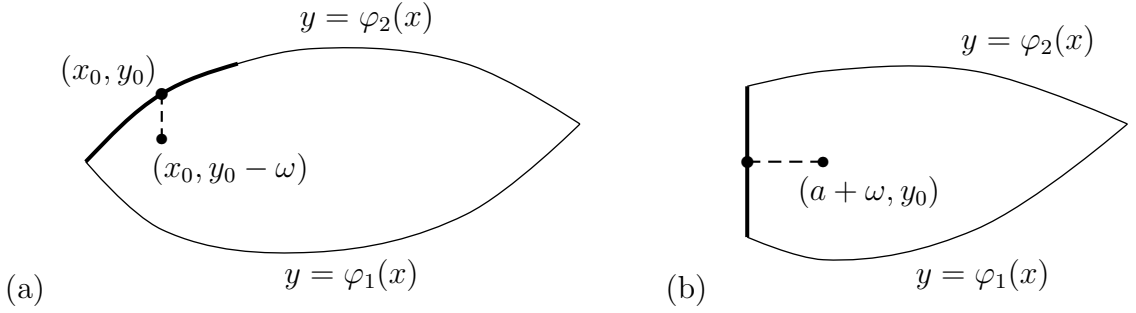
\begin{figure}[h]
\begin{picture}(0,130)
 \scalebox{1}{

 \rput(1.3,2.2){
 \pscurve[linewidth=0.5pt](0,0)(1,0.9)(2,1.3)(3,1.5)(5,1.3)(6.5,0.5)   \psecurve[linewidth=1.5pt](-1,-1.2)(0,0)(1,0.9)(2,1.3)(3,1.5)
 \pscurve[linewidth=0.5pt](0,0)(1,-0.9)(3,-1.2)(5,-0.7)(6.5,0.5)
 \psdots[dotsize=4.5pt](1,0.9)  \psdots[dotsize=4.0pt](1,0.3)
 \psline[linestyle=dashed](1,0.3)(1,0.9)    \rput(0.3,1.2){$(x_0,y_0)$}  \rput(1.9,-0.1){$(x_0,y_0-\om)$}
 \rput(3.7,1.9){$y = \vphi_2(x)$} \rput(3.5,-1.5){$y = \vphi_1(x)$}
  \rput(-0.8,-1.7){(a)}
}

\rput(10,2.2){
\pscurve[linewidth=0.5pt](0,1)(1,1.2)(3,1.2)(5,0.5) \pscurve[linewidth=0.5pt](0,-1)(1,-1.3)(3,-0.9)(5,0.5)
\psline[linewidth=1.5pt](0,1)(0,-1)  \psline[linestyle=dashed](0,0)(1,0)
\psdots[dotsize=4.5pt](0,0)  \psdots[dotsize=4.0pt](1,0) \rput(1.6,-0.4){$(a+\om,y_0)$}
\rput(3.7,1.6){$y = \vphi_2(x)$} \rput(3.1,-1.5){$y = \vphi_1(x)$}
\rput(-0.8,-1.7){(b)}
}

}
\end{picture}
\caption{In figure (a), $\pl\Omega$ does not contain vertical segments, while in figure (b), $\pl\Omega$ contains the vertical segment $x=a, \ y \in [\vphi_1(a),\, \vphi_2(a)]$.}
\label{fig:lim}
\end{figure}

The gradient $\nabla u_\om(x,y)$, with $(x,y)$ on the graphs of $\vphi_1$ and $\vphi_2\big\rfloor_{[a+\del,b]}$ and $\om \in (0,\, \om_0]$, is uniformly bounded. On the other hand, for all regular points $(x,y)$ on the graph of $\vphi_2\big\rfloor_{[a,a+\del]}$, according to Lemma \ref{l.tan}, one has $-\frac{u_x(x,y)}{u_y(x,y)} \in [m-\ve,\, m]$ and in particular, $-\frac{u_x(x_0, y_0)}{u_y(x_0, y_0)} \in [m-\ve,\, m]$ and
$$
-1 = u(x_0, y_0) - u(x_0, y_0-\om) = \om\, u_y(x_0, y_0),
$$
hence $u_y(x_0, y_0) = -1/\om$ and $u_x(x_0, y_0) \in \frac{1}{\om} [m-\ve,\, m]$.  It follows that as $\om \to 0$, the norms $\| u_x \|_\infty$ and $\| u_y \|_\infty$ tend to infinity and, according to Lemma \ref{l.ess}, both these values are attained on the graph of $\vphi_2\big\rfloor_{(a,a+\del]}$. It follows that the ratio
$\frac{\|u_x\|_\infty}{\|u_y\|_\infty}$ belongs to $[m-\ve,\, m].$

Now assume that one of the vertical lines of support to $\Om$ is not angular. We are going to prove that $K_\infty = \infty$. Consider two cases: \ (i) $m = \infty$ and \ (ii) one of the differences $\vphi_2(a) -\vphi_1(a),\ \vphi_2(b) - \vphi_1(b)$ is positive.

(i) Assume without loss of generality that $\vphi_2'(a) = +\infty.$ Fix $\ve > 0$ and choose $\del > 0$ sufficiently small, so as for all $x \in (a,\, a+\del]$, \  $\vphi'_2(x) \in [1/\ve,\, +\infty)$. Take a regular point $(x_0, y_0)$ on the graph of $\vphi_2\rfloor_{[a,a+\del]}$ and denote by $u = u_{\om}$ the smallest concave function satisfying $u\rfloor_{\pl\Om} = 0$ and $u(x_0, y_0-\om) = 1$. Repeating the above argument, one concludes that for $\om$ small enough, $\frac{\|u_x\|_\infty}{\|u_y\|_\infty} \ge 1/\ve.$
It follows that $K_\infty = \infty$.

(ii) Suppose, without loss of generality, that $\vphi_2(a) - \vphi_1(a) > 0.$ Take $y_0 \in (\vphi_1(a),\, \vphi_2(a))$ and denote by $u = u_{\om}$ the smallest concave function satisfying $u\rfloor_{\pl\Om} = 0$ and $u(a+\om, y_0) = 1;$ see Fig.~\ref{fig:lim}\,(b). In what follows we assume that $0 < \om \le \om_0$ with $\om_0$ sufficiently small, so as $(a+\om_0, y_0)$ is an interior point of $\Om$.

When $(x,y)$ is on $\pl\Om \cap \{ x > a  \}$ and $\om \in (0,\, \om_0]$,\ $\nabla u_\om(x,y)$ is uniformly bounded. On the other hand, for all points $(a,y)$ on the segment $x=a, \ y \in [\vphi_1(a),\, \vphi_2(a)]$ one has $u_x(a,y) = 1/\om$ and $u_y(a,y) = 0$. It follows that $\|u_x\|_\infty = 1/\om \to \infty$ as $\om \to 0$ and $\|u_y\|_\infty$ is bounded, hence $K_\infty = \infty$.

Theorem \ref{th.p=infty} is proved.

{\hfill $\Box$}
\smallskip

\begin{remark}\label{r.infty}
In the case $p = \infty$ one easily determines the solution of Minimax problem \ref{pr.minimax}. Indeed,
$$
 \frac{K_\infty(\Omega)}{k_\infty(\Omega)} \ = \ \sup_{\stackrel{u\in\cU_\Om}{u\ne 0}} \ \frac{\|u_x\|_\infty}{\|u_y\|_\infty} \ \
\sup_{\stackrel{u\in\cU_\Om}{u\ne 0}} \ \frac{\|u_y\|_\infty}{\|u_x\|_\infty}
\ = \ \sup \Big| \frac{dy}{dx} \Big| \ \sup \Big| \frac{dx}{dy} \Big|\ = \ \frac{\sup \big| \frac{dy}{dx} \big|}{\inf \big| \frac{dy}{dx} \big|},
$$
where the suprema and the infimum are taken over all graphs of functions forming $\pl\Om$. Clearly, the minimum of the latter quantity is attained when $\Big| \frac{dy}{dx} \Big|$ is constant, that is, $\Om$ is a parallelogram with equal slopes of the sides. Thus,
$$\inf_\Omega\, \frac{K_\infty(\Omega)}{k_\infty(\Omega)} = 1.$$
In a similar way, the more general relation can be derived for nonzero $\bh_1$ and $\bh_2$,
$$
\inf_\Om\ \frac{\sup_{u} \frac{\|u_{\bh_1}\|_\infty}{\|u_{\bh_2}\|_\infty}}{\inf_{u} \frac{\|u_{\bh_1}\|_\infty}{\|u_{\bh_2}\|_\infty}}\ = \ 1,
$$
and the infimum is attained when $\Om$ is a parallelogram with the sides parallel to $\bh_1 + \bh_2$ and $\bh_1 - \bh_2$.
\end{remark}

\medskip

\section*{Acknowledgements}

The work of AP was supported by the projects UIDB/04106/2025 and CoSysM3, 2022.03091.PTDC, through FCT.

\end{document}